\documentclass[12pt, reqno]{amsart}
\usepackage{amsmath, amstext, amsbsy, amssymb, amscd}

\newtheorem{theorem}{Theorem}[section]
\newtheorem{lemma}[theorem]{Lemma}

\theoremstyle{definition}

\newtheorem{remark}[theorem]{Remark}
\newtheorem{conjecture}[theorem]{Conjecture}

\numberwithin{equation}{section}

{\vskip-\lastskip\medskip
  \noindent
  {\em #1.}\enspace
  }%
{\qed\par\medskip
  }

\begin{document}

\title[On equations $(-1)^{\alpha}p^x+(-1)^{\beta}(2^k(2p+1))^y=z^2$ with Sophie Germain prime $p$]{On equations $(-1)^{\alpha}p^x+(-1)^{\beta}(2^k(2p+1))^y=z^2$ with Sophie Germain prime $p$}
\author[ Yuan Li $^{1*}$, Jing Zhang $^{2}$,  Baoxing Liu $^{3}$]{Yuan Li $^{1*}$ , Jing Zhang $^{2}$, Baoxing Liu $^{3}$}

\address{{\small $^{1}$ Mathematics Department, Winston-Salem State University, NC
27110,USA}\\ \newline
{\small email: liyu@wssu.edu}\\ \newline
{\small $^{2}$ Mathematics department, Governors State University, IL
60484,USA}\\ \newline
{\small email: jzhang@govst.edu}\\  \newline
{\small $^{3}$ Mathematics Department, Winston-Salem State University, NC
27110,USA}\\ \newline
{\small email: bliu121@rams.wssu.edu}\\}

\thanks{ $^{*}$ Corresponding author
\\
}

\date{}

\begin{abstract}
In this paper, we consider the Diophantine equation  $(-1)^{\alpha}p^x+(-1)^{\beta}(2^k(2p+1))^y=z^2$ for Sophie Germain prime $p$ with $\alpha, \beta \in\{0,1\}$,  $\alpha\beta=0$ and $k\geq 0$. First, for $p=2$, we solve three  Diophantine equations $(-1)^{\alpha}2^x+(-1)^{\beta}(2^{k} 5)^y=z^2$  by using Nagell-Lijunggren Equation and the database  LMFDB of elliptic curve $y^2=x^3+ax+b$ over $\mathbb{Q}$. Then we obtain all  non-negative integer solutions for the following four types of equations for odd Sophie Germain prime $p$:

i) $p^x+(2^{2k+1}(2p+1))^y=z^2$ with $p\equiv 3, 5 \pmod 8$ and $k\geq 0$;

ii) $p^x+(2^{2k}(2p+1))^y=z^2$ with $p\equiv 3 \pmod 8$ and $k\geq 1$;

iii) $p^x-(2^{k}(2p+1))^y=z^2$ with $p\equiv 3 \pmod 4$ and $k\geq 0$;

iv) $-p^x+(2^{k}(2p+1))^y=z^2$ with $p\equiv 1, 3, 5 \pmod 8$ and $k\geq 1$;

 For each type of the equations, we show the existences of such prime $p$. Since it was conjectured that there exist infinitely many  Sophie Germain primes  in literature,  it is reasonable to conjecture  that  there exist  infinite Sophie Germain primes $p$ such that $p\equiv k \pmod 8$ for any $k\in\{1,3,5,7\}$.  
\end{abstract}
\maketitle


\section{Introduction}
In recent decades, a lot of research has been done on the Diophantine equation $a^x+b^y=z^2$, in particular, for various fixed pairs of $(a, b)$.  P. B. Borah and M. Dutta showed 
that there is only one solution for the equation $7^x+32^y=z^2$ 
\cite{Bor}
and only two solutions for the equation $2^x+7^y=z^2$ for $x\neq 1$. In \cite {Zhang}, the first two authors answered  a question proposed by P. B. Borah and M. Dutta in \cite{Bor} by obtaining all  non-negative solutions of the equation $2+7^y=z^2$.  W. S. Gayo and J. B. Bacani \cite{Gay} solved the equation $M_p^x+(M_q+1)^y=z^2$ for Mersenne primes $M_p=2^p-1$ and $M_q=2^q-1$.  R. J. S. Mina and J. B. Bacani studied the equation $p^x+(p+4k)^y=z^2$ for prime pairs $p$ and $p+4k$ and received many results in \cite{Min2}. A summary of  most recent works on the equation $a^x+b^y=z^2$ can be found in \cite{Gay, Min2}. On the other hand, the equation $a^x-b^y=z^2$ also received a lot of attention \cite{Le, Ter, Yua2}. 

A prime $p$ is a Sophie Germain prime if $2p+1$ is also a prime. Sophie Germain primes  are named after French mathematician Sophie Germain, who used them in her investigation of Fermat's Last Theorem \cite{Mus}. A prime $2p+1$ is called a safe prime if $p$ is a prime. Sophie Germain primes and safe primes have applications in public key cryptography and primality testing \cite{Gat}. 
 
The following theorem was  stated by Euler in 1750 and proved by Lagrange in 1775: If $p\equiv 3$ (mod 4) is a prime, then $2p+1$ is also prime if and only if $2p+1$ divides $2^p-1$. This theorem shows the relation between Sophie Germain primes and Mersenne primes. It has been conjectured that there are infinite many Sophie Germain (Safe) primes \cite{Sh}. 
In fact,  on page 123 in  Shoup's book,  it is  commented:   ``numerical evidence, and heuristic arguments, strongly
suggest not only that there are infinitely many such primes, but also a fairly precise
estimate on the density of such primes.''  

In this paper,  we consider Diophantine equations $(-1)^{\alpha}p^x+(-1)^{\beta}(2^k(2p+1))^y=z^2$ for the prime pairs $(p, 2p+1)$, where $\alpha, \beta =0,1$ and $k\geq 0$.   We will list some results which are needed in the proof of our main theorems in Section 2.
We will first solve three equations $(-1)^{\alpha}2^x+(-1)^{\beta}(2^k 5)^y=z^2$ by using Nagell-Lijunggren Equation and the  the database of LMFDB on elliptic curves over $\mathbb{Q}$ in Section 3. Then we solve  four types of  equations $(-1)^{\alpha}p^x+(-1)^{\beta}(2^k(2p+1))^y=z^2$ for odd Sophie Germain prime $p$ in Section 4.

\section{Preliminaries}

Conjectured by  Eug$\grave{\mbox{e}}$ne Charles Catalan in 1844,
   Lemma \ref{ThCM} was  proved by  by Preda Mih$\breve{\mbox{a}}$ilescu \cite{Mih}
in 2002.

\begin{lemma} [\bf{Catalan-Mih$\breve{\mbox{a}}$ilescu Theorem}]\label{ThCM}
 The unique solution for the Diophantine equation $a^x-b^y=1$ where $a$, $b$, $x$, $y$ $\in\mathbb{Z}$ with $\min\{a,b,x,y\}>1$ is $(a,b,x,y)=(3,2,2,3)$.
\end{lemma}

Lemma \ref{lm1}  is a direct consequence of \cite{Ko} and a special case of  Catalan-Mih$\breve{\mbox{a}}$ilescu Theorem. One can also prove it directly.
\begin{lemma}\label{lm1}
If $p$ is a prime, then $p^x+1=y^2$ has only non-negative solutions $(p,x,y)=(2,3,3)$ and $(p,x,y)=(3,1,2)$.
\end{lemma}

Ljunggren and  Nagell  proved the following results \cite{Bug, Lju}.

\begin{lemma} [\bf{Nagell-Ljunggren Theorem}]\label{ThNL}
If $x,y>1$, $n>2$, $q\geq 2$, apart from the solutions $\frac{3^5-1}{3-1}=11^2$, $\frac{7^4-1}{7-1}=20^2$ and $\frac{18^3-1}{18-1}=7^3$ the following equation
\begin{equation}\label{eq1}
\frac{x^n-1}{x-1}=y^q
\end{equation}
has no other solution $(x,y,n,q)$ if either one of the following conditions is satisfied:

(i) $q=2$,

(ii) $3$ divides $n$,

(iii) $4$ divides $n$,

(iv) $q=3$ and $n\not\equiv 5$ $\pmod 6$.
\end{lemma}

More works about Nagell-Ljunggren equation can be found in \cite{Bug0, Bug, Li, Yua}.

The discriminant of elliptic curve $y^2=x^3+ax+b$ over $\mathbb{Q}$ is $\Delta=-16(4a^3+27b^2)$. 
Using discriminant to search in 
``The L-functions and modular forms database'' (LMFDB) \cite{LM},
 we  have 

\begin{lemma}\label{lmCurve}
The three elliptic curves over $\mathbb{Q}$ and their integral points are given below.

(1)  $y^2=x^3-4$, $(x,y)=(2,\pm 2)$, $(5,\pm 11)$. 

(2)  $y^2=x^3-100$, $(x,y)=(5,\pm 5)$, $(10,\pm 30)$, $(34,\pm 198)$.

(3)  $y^2=x^3-2500$, $(x,y)=(50,\pm 350)$.

\end{lemma}

The equations in Lemma \ref{lmCurve} are examples of Mordell's equation. A  Mordell's equation   has only
finitely many integral solutions \cite{C}.

The following simple facts will be useful.

If $2^k-1$ is a prime, then $k$ is a prime. If $2^k+1$ is a prime, then $k=0$ or $k=2^r$ with $r\geq 0$.

\begin{lemma}\label{lmk}
If $p$ is a Sophie Germain prime and $k\geq 1$, then Diophantine equation $1+(2^k(2p+1))^y=z^2$ has only non negative solutions $(p,k,y,z)=(2,4,1,9), (3,5,1,15)$.
\end{lemma}

\begin{proof}
In $z^2-(2^k(2p+1))^y=1$, obviously, $z\geq 2$. By Lemma \ref{ThCM}, we get $y=1$. Let $q=2p+1$, we obtain

$$(z-1)(z+1)=2^kq.$$

Let $d=\gcd(z-1,z+1)$, then $d\mid 2$,  we obtain $d=2$ since $k\geq 1$. There are three possibilities. Namely, $z-1=2, 2q, 2^{k-1}$ since $z-1<z+1$.

Case 1: $z-1=2$ and $z+1=2^{k-1}q$.

In this case, we have $2=2^{k-1}q-2$ which is impossible.

Case 2: $z-1=2q$ and $z+1=2^{k-1}$.

In this case, we have $2=2^{k-1}-2q$. We have $q=2^{k-2}-1$ and $p=2^{k-3}-1$. Since both $p$ and $q$ are primes, so are $k-2$ and $k-3$. Therefore, $k=5$ and $(p,k,y,z)=(3,5,1,15)$.

Case 3: $z-1=2^{k-1}$ and $z+1=2q$.

In this case, we have $q=2^{k-2}+1$ and $p=2^{k-3}$. Since  $p$ is prime, so  $p=2$ and $k=4$. Therefore, $(p,k,y,z)=(2,4,1,9)$.
\end{proof}

For $k=1, 3, 5, 7$, let ${\mathbb{S}\mathbb{G}}_k$ be the set of Sophie Germain primes $p$, such that $p\equiv k \pmod 8$. We list all  elements which are less than 1000 in ${\mathbb{S}\mathbb{G}}_k$  below. 

\[
{\mathbb{S}\mathbb{G}}_1=\{41, 89, 113, 233, 281, 593, 641, 761, 809, 953,\dots \dots\},
\]

\[
{\mathbb{S}\mathbb{G}}_3=\{3, 11, 83, 131, 179, 251, 419, 443, 491, 659, 683,\dots \dots\},
\]

\[
{\mathbb{S}\mathbb{G}}_5=\{5, 29, 53, 173, 293, 509, 653,\dots \dots\},
\]

\[
{\mathbb{S}\mathbb{G}}_7=\{23, 191, 239, 359, 431, 719, 743, 911,\dots \dots\}.
\]

It is reasonable to conjecture that each of the above sets is infinite. 
In general, one can have the following.
\begin{conjecture}
For any $m\in \mathbb{N}$, and any $k=1,3,\dots, 2^m-1$, there are infinite many Sophie Germain prime $p$ such that $p\equiv k \pmod {2^m}$. 
\end{conjecture}

\section{Solutions of the equations $(-1)^{\alpha}2^x+(-1)^{\beta}(2^{k}5)^y=z^2$}

In this section, we will solve three Diophantine equations 
$(-1)^{\alpha}2^x+(-1)^{\beta}(2^{k}5)^y=z^2$. We need two lemmas.

The first lemma was proved by Tomita \cite{To}. 
 For the completeness, we will provide the detailed proof.
\begin{lemma}\label{lm5.1}
The Diophantine equation $5^x=4+y^2$ has only non-negative solutions $(x,y)=(1,1), (3,11)$.
\end{lemma}

\begin{proof}
We consider three cases:

Case 1. $x=3k$.

Let $X=5^k$, the original equation can be written as

$$y^2=X^3-4.$$

By Lemma \ref{lmCurve}, we obtain $(X,y)=(2,\pm 2), (5,\pm 11)$. Hence, $(x,y)=(3,11)$.
\\

Case 2. $x=3k+1$.

Let $X=5^{k+1}$, $Y=5y$, the original equation can be written as

$$Y^2=X^3-100.$$

By Lemma \ref{lmCurve}, we obtain $(X,Y)=(5,\pm 5), (10,\pm 30), (34,\pm 198)$. 

Hence,  $(x,y)=(1,1)$.
\\

Case 3. $x=3k+2$.

Let $X=5^{k+2}$, $Y=25y$, the original equation can be written as

$$Y^2=X^3-2500.$$

By Lemma \ref{lmCurve}, we obtain $(X,Y)=(50,\pm 350)$. 

Hence,  no solution for $x,y$.

\end{proof}

Similarly, we have
\begin{lemma}\label{lm5.2}
The Diophantine equation $2( 5^x)=1+y^2$ has only non-negative solutions $(x,y)=(0,1), (1,3), (2,7)$.
\end{lemma}

\begin{proof}
We consider three cases:

Case 1. $x=3k$.

Let $X=2\times 5^k$, $Y=2y$, the original equation can be written as

$$Y^2=X^3-4.$$

By Lemma \ref{lmCurve}, we obtain $(X,Y)=(2,\pm 2), (5,\pm 11)$. Hence, $(x,y)=(0,1)$.
\\

Case 2. $x=3k+1$.

Let $X=2\times 5^{k+1}$, $Y=10y$, the original equation can be written as

$$Y^2=X^3-100.$$

By Lemma \ref{lmCurve}, we obtain $(X,Y)=(5,\pm 5), (10,\pm 30), (34,\pm 198)$. 

Hence,  $(x,y)=(1,3)$.
\\

Case 3, $x=3k+2$.

Let $X=2\times 5^{k+2}$, $Y=50y$, the original equation can be written as

$$Y^2=X^3-2500.$$

By Lemma \ref{lmCurve}, we obtain $(X,Y)=(50,\pm 350)$. 

Hence,  $(x,y)=(2,7)$.
\end{proof}

In 2007, Dumitru Acu \cite{Acu} solved the equation $2^x+5^y=z^2$. It has only solutions $(x,y,z)=(2,1,3),(3,0,3)$. We have

\begin{theorem}\label{4.1}
For $k\geq 1$, the Diophantine equation 
\begin{equation}\label{eq4.1}
2^x+(2^{k}5)^y=z^2
\end{equation}
 has only non-negative solutions $(x,y,z)=(3,0,3)$ and 

$(k,x,y,z)=(2n,2n+2,1,2^n+2^{n+1})$, $(2n+2,2n-2,1,2^{n-1}+2^{n+2})$ for any $n\in \mathbb{N}$.

\end{theorem}

\begin{proof}
If $y=0$, we have $(x,y,z)=(3,0,3)$ by Lemma \ref{lm1}.

For $y\geq 1$, by Equation \ref{eq4.1}, we have $2^x\equiv z^2 \pmod 5$. Therefore $(\frac{2}{5})^x=(-1)^x=1$ and $x=2m$. We rewrite Equation \ref{eq4.1} as 

$$(z-2^m)(z+2^m)=2^{ky}5^y.$$

Obviously, both $z-2^m$ and $z+2^m$ are even since $ky\geq 1$. Let $d=\gcd(z-2^m, z+2^m)$, then $d\mid 2^{m+1}$ and $d^2\mid 2^{ky}5^y$. So, $d=2^n$, $1\leq n\leq \min \{m+1, \lfloor\frac{ky}{2}\rfloor \}.$

We have the following three possibilities since $z-2^m<z+2^m$ and $n\leq ky-n$.

Case 1: $z-2^m=2^n$ and $z+2^m=2^{ky-n}5^y$.

In this case, we obtain $2^{m+1}=2^{ky-n}5^y-2^n$, hence, $2^{m+1-n}=2^{ky-2n}5^y-1$. It is clear that $m+1-n=0$ is impossible. For $m+1-n\geq 1$, $2^{m+1-n}$ is even so is $2^{ky-2n}5^y-1$. Therefore, we have $ky-2n=0$. So, we obtain
$2^{m+1-n}=5^y-1$ or $5^y-2^{m+1-n}=1$. By Lemma \ref{ThCM}, we get $y=1$ and $m+1-n=2$. Hence, $(k,x,y,z)=(2n,2n+2,1,2^n+2^{n+1})$.

Case 2: $z-2^m=2^{ky-n}$ and $z+2^m=2^{n}5^y$.

In this case, we have $2^{m+1-n}=5^y-2^{ky-2n}$.

If $m+1-n=0$, then $5^y-2^{ky-2n}=1$. By Lemma \ref{ThCM}, we get $y=1$ and $ky-2n=2$. Hence, $(k,x,y,z)=(2n+2,2n-2,1,2^{n-1}+2^{n+2}).$

If $m+1-n\geq 1$, then $ky=2n$ since $5^y-2^{ky-2n}$ must be even. Exactly same to Case 1, we obtain the solution of Case 1 again.

Case 3: $z-2^m=2^n5^y$ and $z+2^m=2^{ky-n}$.

In this case, we have $2^{m+1-n}=2^{ky-2n}-5^y$.

If $m+1-n=0$, then $2^{ky-2n}-5^y=1$ which is impossible by Lemma \ref{ThCM}.

If $m+1-n\geq 1$, then $ky=2n$ since $2^{ky-2n}-5^y$ must be even. We have $5^y+2^{m+1-n}=1$ which is absurd.

\end{proof}

\begin{theorem}\label{4.2}
For $k\geq 0$, the Diophantine equation 
\begin{equation}\label{eq4.2}
2^x-(2^{k}5)^y=z^2
\end{equation}
 has only non-negative solutions $(x,y,z)=(0,0,0), (1,0,1).$ 
\end{theorem}

\begin{proof}
For $x=0,1$, we get two solutions $(x,y,z)=(0,0,0), (1,0,1).$

We show Equation \ref{eq4.2} has no solutions when $x\geq 2$.

If $y=0$, we have $-1\equiv z^2 \pmod 4$ which is impossible.

For $y\geq 1$, we have $2^x\equiv z^2 \pmod 5$. So, we get $(\frac{2}{5})^x=(-1)^x=1$ and $x=2m$ with $m\geq 1$. We rewrite Equation \ref{eq4.2} as 

$$(2^m-z)(2^m+z)=2^{ky}5^y.$$

Let $d=\gcd(2^m-z, 2^m+z)$, then $d\mid 2^{m+1}$ and $d^2\mid 2^{ky}5^y$. So, $d=2^n$, $0\leq n\leq \min \{m+1, \lfloor\frac{ky}{2}\rfloor \}.$

We consider the following three possibilities.

Case 1: $2^m-z=2^n$ and $2^m+z=2^{ky-n}5^y$.

In this case, we obtain $2^{m+1}=2^{ky-n}5^y+2^n$, hence, $2^{m+1-n}=2^{ky-2n}5^y+1$. It is clear that $m+1-n=0$ is impossible. For $m+1-n\geq 1$, $2^{m+1-n}$ is even and so is $2^{ky-2n}5^y+1$. Therefore, we have $ky-2n=0$. So, we obtain
$2^{m+1-n}=5^y+1$ or $2^{m+1-n}-5^y=1$ which has no solutions by Lemma \ref{ThCM} since $y\geq 1$.

Case 2:  $2^m-z=2^{ky-n}$ and $2^m+z=2^{n}5^y$.

The proof is almost identical to case 1. There is no solutions.

Case 3: $2^m-z=2^{n}5^y$ and $2^m+z=2^{ky-n}$.

The proof is almost identical to case 1. There is no solutions.

\end{proof}

\begin{theorem}\label{4.3}
For $k\geq 0$, the Diophantine equation 
\begin{equation}\label{eq4.3}
-2^x+(2^{k}5)^y=z^2
\end{equation}
 has only non-negative solutions $(x,y,z)=(0,0,0)$ and 

$(k,x,y,z)=(n-1, 2n+2, 2, 3\times 2^{n-1})$, $(2n-1, 2n-2, 1, 3\times 2^{n-1})$, $(2n-2, 2n-2, 1, 2^n)$, $(2n-2, 2n, 1, 2^{n-1})$, $(2n-2, 6n-4, 3, 11\times 8^{n-1})$ for any $n\in \mathbb{N}$.

\end{theorem}

\begin{proof}
If $y=0$, then $(x,y,z)=(0,0,0)$.

We consider $y\geq 1$ in the following.  

Case 1: $2\mid y$.

Let $y=2m$, $m\geq 1$. We rewrite Equation \ref{eq4.3} as

$$((2^k5)^m-z)((2^k5)^m+z)=2^x.$$

Since $m\geq 1$, then $z\neq 0$, we have $(2^k5)^m-z<(2^k5)^m+z$. 

Therefore, we obtain $(2^k5)^m-z=2^t$ and $(2^k5)^m+z=2^{x-t}$ with $0 \leq t<x-t$. Adding these two equations, we receive 
$2^{km+1-t}5^m=1+2^{x-2t}$. The right side of this equation is odd, so we have $km+1-t=0$. Hence, $5^m=1+2^{x-2t}$, or $5^m-2^{x-2t}=1$. 
By Lemma \ref{ThCM}, we obtain $m=1$ and $x-2t=2$. So, $(k, x, y, z)=(t-1, 2t+2, 2, 3\times 2^{t-1})$ with $t\geq 1$. Let $n=t$, we have
$(k, x,y,z)=(n-1, 2n+2, 2, 3\times 2^{n-1})$.

Case 2: $2\nmid y$.

Let $z=2^tz_1$, $t\geq 0$ and $2\nmid z_1$.

Case 2.1: $ky>x$.

In this sub case, we rewrite Equation \ref{eq4.3} as 

$$2^x(2^{ky-x}5^y-1)=2^{2t}z_1^2.$$

We receive $x=2t$ and $2^{ky-x}5^y-1=z_1^2.$ If $ky-x\geq 2$, we get a contradiction by taking modulo 4. So, $ky-x=1$. Then we have 
$2\times 5^y=1+z_1^2$. By Lemma \ref{lm5.2} and $2\nmid y$, we get $(y, z_1)=(1, 3)$. So, $(k,x,y,z)=(2t+1,2t,1,3\times 2^t)$. Let $n=t+1$, we get $(k,x,y,z)=(2n-1, 2n-2, 1, 3\times 2^{n-1})$.

Case 2.2: $ky=x$.

In this sub case, we rewrite Equation \ref{eq4.3} as 

$$2^{x+2}(5^{y-1}+\cdots+5+1)=2^{2t}z_1^2.$$

Since $2\nmid y$, we have $x+2=2t$ and $5^{y-1}+\cdots+5+1=z_1^2$. So, we get 

$$\frac{5^y-1}{5-1}=z_1^2.$$

By Lemma \ref{ThNL}, we know $y=1$ since $y$ is odd. Hence, $(k,x,y,z)=(2t-2,2t-2,1,2^t)$ with $t\geq 1$. Let $n=t$, we obtain 
$(k,x,y,z)=(2n-2, 2n-2, 1, 2^n)$.

Case 2.3: $ky<x$.

In this sub case, we rewrite Equation \ref{eq4.3} as 

$$2^{ky}(5^y-2^{x-ky})=2^{2t}z_1^2.$$

We obtain $ky=2t$ and 

$$5^y-2^{x-ky}=z_1^2.$$

Case 2.3.1:  $x-ky=1$.

We obtain $5^y=2+z_1^2$. Since $y\geq 1$, we have $z_1^2\equiv -2 \pmod 5$ which is impossible due to $(\frac{-2}{5})=-1$.

Case 2.3.2:  $x-ky=2$.

We obtain $5^y=4+z_1^2$. By Lemma \ref{lm5.1}, we have $(y,z_1)=(1,1),(3,11)$.

For $(y,z_1)=(1,1)$, we obtain $(k,x,y,z)=(2t,2t+2,1,2^t)$ with $t\geq 0$. Let $n=t+1$, we get $(k,x,y,z)=(2n-2, 2n, 1, 2^{n-1})$.

For $(y,z_1)=(3,11)$, from $ky=3k=2t$, we know $3\mid t$. Let $t=3t_1$, we get $(k,x,y,z)=(2t_1,6t_1+2,3,11\times 8^{t_1})$ with $t_1\geq 0$.
 Let $n=t_1+1$, we obtain $(k,x,y,z)=(2n-2, 6n-4, 3, 11\times 8^{n-1})$.

Case 2.3.3: $x-ky\geq 3$.

From $5^y-2^{x-ky}=z_1^2$, we get $5^y\equiv z_1^2 \pmod 8$. Since $y$ is odd, we obtain $5\equiv z_1^2 \pmod 8$ which is absurd.
\end{proof}

\section{On the equations $(-1)^{\alpha}p^x+(-1)^{\beta}(2^k(2p+1))^y=z^2$ for odd Sophie Germain prime $p$}

In this section, we study  four equations $(-1)^{\alpha}p^x+(-1)^{\beta}(2^k(2p+1))^y=z^2$ for odd Sophie Germain prime $p$. We have
\begin{theorem}\label{th1}
If $p\equiv 3, 5 \pmod 8$ is a Sophie Germain prime and $k\geq 0$, then equation
\begin{equation}\label{eq3.1}
p^x+(2^{2k+1}(2p+1))^y=z^2
\end{equation}
has only non-negative solutions $(p, x,y,z)=(3,1,0,2)$, $(p,k,x,y,z)=(3,2,0,1,15)$.
\end{theorem}
\begin{proof}
Let $q=2p+1$, then $q$ is a prime and greater than or equal to 7.

When $x=0$,  by Lemma \ref{lmk}, Equation \ref{eq3.1} has only solution $(p,2k+1,y,z)=(3,5,1,15)$, so, $(p,k,x,y,z)=(3,2,0,1,15)$.

When $y=0$, by Lemma \ref{lm1}, $(p,x,z)=(3,1,2)$. So, $(p, x,y,z)=(3,1,0,2)$.

Now, we assume $x, y \geq 1$. In Equation \ref{eq3.1}, by taking modulo $p$, we get

\[
(2^{2k+1})^y\equiv z^2 \pmod p.
\]

We have $(\frac{(2^{2k+1})^y}{p})=(\frac{2^y}{p})=(-1)^y=(\frac{z^2}{p})=1.$

So, $y=2n$, $n\geq 1$. Equation \ref{eq3.1} can be written as

$$(z-(2^{2k+1}q)^n)(z+(2^{2k+1}q)^n)=p^x.$$

Let $d=\gcd(z-(2^{2k+1}q)^n, z+(2^{2k+1}q)^n)$, then $d\mid 2(2^{2k+1}q)^n=2^{(2k+1)n+1}q^n$, $d\mid p^x$. So, $d=1$.

We obtain $(z-(2^{2k+1}q)^n)=1$ and $(z+(2^{2k+1}q)^n)=p^x$. Therefore,

$$2(2^{2k+1}q)^n=p^x-1.$$

If $k=0$ and $n=1$, we obtain $4q=p^x-1$, then $4\equiv -1 \pmod p$, hence, $p=5$ and $45=5^x$ which is impossible.

So, we have either $k\geq 1$ or $n\geq 2$. Taking modulo 8 in $2(2^{2k+1}q)^n=p^x-1$, we obtain $p^x\equiv 1 \pmod 8$. So, 
$x=2m$, $m\geq 1$. We rewrite $2(2^{2k+1}q)^n=p^x-1$ as

$$(p^m-1)(p^m+1)=2^{(2k+1)n+1}q^n.$$

There are three possibilities since $\gcd(p^m-1, p^m+1)=2$.

Case 1: $p^m-1=2$ and $p^m+1=2^{(2k+1)n}q^n$.

In this case, we get $2=2^{(2k+1)n}q^n-2$ which is absurd.

Case 2: $p^m-1=2q^n$ and $p^m+1=2^{(2k+1)n}$.

In this case, we get $2^{(2k+1)n-1}-q^n=1$. By Lemma \ref{ThCM}, we have $n=1$. So, $q=2^{2k}-1$, $2k$ must be a prime. We obtain $k=1$ and $p=1$ which is absurd.

Case 3: $p^m-1=2^{(2k+1)n}$ and $p^m+1=2q^n$.

In this case, we get $q^n-2^{(2k+1)n-1}=1$ with $q\geq 7$. Again, by Lemma \ref{ThCM}, we have $n=1$ and $q=2^{2k}+1$. Hence, $p=2^{2k-1}$ which is a contradiction since $p$ is an odd prime.
\end{proof}

\begin{theorem}\label{th2}
If $p\equiv 3 \pmod 8$ is a Sophie Germain prime and $k\geq 1$, then equation
\begin{equation}\label{eq3.2}
p^x+(2^{2k}(2p+1))^y=z^2
\end{equation}
has only non-negative solutions $(p, x,y,z)=(3,1,0,2)$, $(p,k,x,y,z)=(3,2,6,1,29),$
$(3,2,2,1,11), (3,3,4,1,23)$.
\end{theorem}
\begin{proof}
Let $q=2p+1$, then $q$ is a prime and greater than or equal to 7.

When $x=0$,  by Lemma \ref{lmk}, Equation \ref{eq3.2} has no solution since $2k\neq 5$ and $p\neq 2$.

When $y=0$, by Lemma \ref{lm1},  $(p, x,y,z)=(3,1,0,2)$.

Now, we assume $x, y \geq 1$. In Equation \ref{eq3.2}, by taking modulo $q$, we get

\[
p^x\equiv z^2 \pmod q.
\]

Since $(\frac{p}{q})=(\frac{q}{p})(-1)^{\frac{p-1}{2}p}=(\frac{1}{p})(-1)^{\frac{p-1}{2}p}=-1$. We obtain $(-1)^x=1$ and $x=2m$ with $m\geq1$. Equation \ref{eq3.2} can be written as 

$$(z-p^m)(z+p^m)=2^{2ky}q^y.$$

We consider three cases.

Case 1: $z-p^m=2$ and $z+p^m=2^{2ky-1}q^y$.

In this case, we get $p^m=2^{2ky-2}q^y-1$. We consider three sub cases.

Case 1.1: $2ky-2\geq 3$.

We receive $p^m\equiv -1\pmod 8$ which is impossible since $p\equiv 3 \pmod 8$.

Case 1.2: $2ky-2=2$.

We have $ky=2$. There are two possibilities.

Case 1.2.1: $k=1$ and $y=2$.

We obtain $p^m=4q^2-1$, hence, $0\equiv 4-1 \pmod p$. So, $p=3$, $q=7$ and $3^m=4\times 7^2-1=195=3\times 5\times 13$ which is absurd.

Case 1.2.2: $k=2$ and $y=1$.

We obtain $p^m=4q-1=8p+3$, take modulo $p$ we know $p=3$. Hence, $3^m=27$, $m=3$, we get another solution $(p,k,x,y,z)=(3,2,6,1,29)$.

Case 1.3: $2ky-2=0$.

We have $k=y=1$ and $p^m=q-1=2p$ which is impossible since $p$ is odd.

Case 2: $z-p^m=2q^y$ and $z+p^m=2^{2ky-1}$.

In this case, we get $p^m=2^{2ky-2}-q^y$. Take modulo $q$ we obtain

$$p^m\equiv (2^{ky-1})^2 \pmod q.$$

So, we get $(-1)^m=1$ by computing the legendre symbol. Let $m=2m_1$ with $m_1\geq 1$. We rewrite $p^m=2^{2ky-2}-q^y$ as

$$(2^{ky-1}-p^{m_1})(2^{ky-1}+p^{m_1})=q^y.$$

So, we get $2^{ky-1}-p^{m_1}=1$ and $2^{ky-1}+p^{m_1}=q^y$. Adding these two equations, we receive $2^{ky}=q^y+1$, or $2^{ky}-q^y=1$. By Lemma \ref{ThCM}, $y=1$. Then we have $q=2^k-1$ and $p=2^{k-1}-1$. Both $p$ and $q$ are primes, so are $k-1$ and $k$. Therefore, $k-1=2$, $k=3$. With simple calculation we obtain another solution $(p,k,x,y,z)=(3,3,4,1,23)$.

Case 3: $z-p^m=2^{2ky-1}$ and $z+p^m=2q^y$. 

In this case, we get 

$$p^m=q^y-2^{2ky-2}$$. 

Take modulo $q$, we obtain $p^m\equiv -(2^{ky-1})^2 \pmod q.$

Computing the legendre symbol, we obtain $(-1)^m=-1$, so $2\nmid m$. We consider three sub cases.

Case 3.1: $2ky-2\geq 3$.

Take modulo 8 in equation $p^m=q^y-2^{2ky-2}$, we have $p^m\equiv q^y \pmod 8$. However, $p^m\equiv p\equiv 3 \pmod 8$ but $q^y\equiv 7^y\equiv 1,7 \pmod 8$. We receive a contradiction.

Case 3.2: $2ky-2=2$.

We have $ky=2$. There are two possibilities.

Case 3.2.1: $k=1$ and $y=2$.

We have $p^m=q^2-4$, $0\equiv 1-4 \pmod p$. So, $p=3$ and $3^m=45$ which is absurd.

Case 3.2.2: $k=2$ and $y=1$.

We have $p^m=q-4=2p-3$. So, $p=3$, $m=1$. We receive another solution $(p,k,x,y,z)=(3,2,2,1,11)$.

Case 3.3: $2ky-2=0$.

We obtain $k=y=1$, then $p^m=q-1=2p$ which is impossible.
\end{proof}

\begin{remark}\label{re3.1}
When $p\equiv 3 \pmod 4$ and $k=0$, the equation $p^x+(2^{2k}(2p+1))^y=z^2$ becomes $p^x+(2p+1)^y=z^2$. We did not include this case in Theorem \ref{th2}. The reason is it is a special case of the equation $p^x+(p+4k)^y=z^2$ where both $p$ and $p+4k$ are primes by letting $p+1=4k$. The equation $p^x+(p+4k)^y=z^2$ was discussed in \cite{Min2}.
\end{remark}

\begin{theorem}\label{th3}
If $p\equiv 3 \pmod 4$ is a Sophie Germain prime and $k\geq 0$, then equation
\begin{equation}\label{eq3.3}
p^x-(2^k(2p+1))^y=z^2
\end{equation}
has only non-negative solutions $(x,y,z)=(0,0,0)$, $(p,k,x,y,z)=(3,3,4,1,5)$.
\end{theorem}
\begin{proof}
Let $q=2p+1$. 

If $x=0$, then $(x,y,z)=(0,0,0)$.

For $x\geq 1$, take modulo $p$ in Equation \ref{eq3.3}, we obtain

$$-(2^k)^y\equiv z^2 \pmod p.$$

So, $1=(\frac{-2^{ky}}{p})=-(\frac{2}{p})^{ky}$. We must have $2 \nmid ky$ and $p\equiv 3 \pmod 8$. Since $y$ is odd, then $y\geq1$. Take modulo $q$ in Equation \ref{eq3.3}, we have $p^x\equiv z^2 \pmod q$. Because $(\frac{p}{q})=-1$, we obtain $(-1)^x=1$. Let $x=2m$ with $m\geq 1$. We rewrite Equation \ref{eq3.3} as

$$(p^m-z)(p^m+z)=2^{ky}q^y.$$

Since $\gcd(p^m-z,p^m+z)=2$, we consider three cases.

Case 1: $p^m-z=2$ and $p^m+z=2^{ky-1}q^y$.

In this case, we have $p^m=2^{ky-2}q^y+1$. Then, $p^m\equiv 1 \pmod q$. So, $m=2m_1$ for $m_1\geq 1$. So the equation $p^m=2^{ky-2}q^y+1$ can be written as 

$$(p^{m_1}-1)(p^{m_1}+1)=2^{ky-2}q^y.$$

There are three sub cases.

Case 1.1: $p^{m_1}-1=2$ and $p^{m_1}+1=2^{ky-3}q^y$.

In this sub case, $2=2^{ky-3}q^y-2$ which is impossible since $y\geq 1$.

Case 1.2: $p^{m_1}-1=2q^y$ and $p^{m_1}+1=2^{ky-3}$.

In this sub case, we have $1=2^{ky-4}-q^y$. By Lemma \ref{ThCM}, we obtain $y=1$. Hence $q=2^{k-4}-1$ and $p=2^{k-5}-1$. Since both $p$ and $q$ are primes, so are  $k-4$ and $k-5$. Hence, $k-5=2$, $k=7$. So, $p=3$ and $3^{m_1}=1+2q=15$ which is absurd.

Case 1.3: $p^{m_1}-1=2^{ky-3}$ and $p^{m_1}+1=2q^y$.

In this sub case, we have $1=q^y-2^{ky-4}$. By Lemma \ref{ThCM}, we obtain $y=1$ since $q>3$. Hence $q=2^{k-4}+1$ and $p=2^{k-5}$ which is impossible. 

Case 2: $p^m-z=2q^y$ and $p^m+z=2^{ky-1}$.

We have $p^m=2^{ky-2}+q^y$. Hence, $p^m\equiv 2^{ky-2} \pmod q$. We get $(-1)^m=(-1)^{ky-2}=-1$ since $ky$ is odd. Therefore, $2\nmid m$.

If $ky-2\geq 3$, then we receive $p^m\equiv q^y \pmod 8$. But $p^m\equiv p\equiv 3 \pmod 8$ and $q^y\equiv q\equiv 7 \pmod 8$, we obtain a contradiction. Hence, $ky-2=1$, or $ky=3$. 

Case 2.1: $k=1$ and $y=3$.

We have $p^m=2+q^3$. Hence, $0\equiv 3 \pmod p$. So, $p=3$ and $3^m=2+7^3=3\times 5 \times 23$ which is absurd.

Case 2.2: $k=3$ and $y=1$.

We have $p^m=2+q=2p+3$. So, $p=3$, $3^m=9$ and $m=2$. But $z=p^m-2q^y=3^2-2(7)^1<0$, so, no solutions.

Case 3: $p^m-z=2^{ky-1}$ and $p^m+z=2q^y$.

At this time, we have $p^m=2^{ky-2}+q^y$. Exactly same to the discussion in Case 2, we have $ky-2=1$, or $ky=3$. 

Case 3.1: $k=1$ and $y=3$.

Exactly same to Case 2.1. No solutions

Case 3.2: $k=3$ and $y=1$.

We have $p^m=2+q=2p+3$. So, $p=3$, $3^m=9$ and $m=2$. Hence, we obtain another solution $(p,k,x,y,z)=(3,3,4,1,5)$.

\end{proof}

\begin{theorem}\label{th4}
If $p\equiv 1, 3, 5 \pmod 8$ is a Sophie Germain prime and $k\geq 1$, then equation
\begin{equation}\label{eq3.4}
-p^x+(2^k(2p+1))^y=z^2
\end{equation}
has only non-negative solutions $(x,y,z)=(0,0,0)$, $(p,k,x,y,z)=(3,1,3,2,13),$
$(3,2,3,1,1), (3,2,1,1,5), (11,2,1,1,9)$.
\end{theorem}

\begin{proof}
Let $q=2p+1$. 

If $y=0$, then $(x,y,z)=(0,0,0)$.

For $y\geq 1$, take modulo $q$ in Equation \ref{eq3.4}, we obtain

$$-p^x\equiv z^2 \pmod q.$$

If $p\equiv 1,5 \pmod 8$, then $q\equiv 3 \pmod 8$. We get $1=(\frac{-p^x}{q})=-1$, a contradiction.

Hence, from now on, we assume  $p\equiv 3 \pmod 8$ and $q\equiv 7 \pmod 8$. From $-p^x\equiv z^2 \pmod q$, we obtain $(-1)^{x+1}=1$. So, $2\nmid x$ and $x\geq 1$. Take modulo $p$ in Equation \ref{eq3.4}, we get $2^{ky}\equiv z^2 \pmod p$. Since $(\frac{2}{p})=-1$, we have $(-1)^{ky}=1$. So, $2\mid ky$. If $ky\geq 4$, we get $-p^x\equiv z^2 \pmod 8$ by Equation \ref{eq3.4}. But $-p^x\equiv -p\equiv 5 \pmod 8$. We receive a contradiction. So, $ky=2$. 

We consider two possibilities.

Case 1: $k=2$ and $y=1$.

Equation \ref{eq3.4} can be written as $-p^x+8p+4=z^2$. We know $p=3,11,\dots$ since $p\equiv 3 \pmod 8$.

For $p=3$, we have $-3^x+28=z^2$, or $28=3^x+z^2$. Hence, $(x,z)=(3,1), (1,5)$. We obtain two solutions $(p,k,x,y,z)=(3,2,3,1,1), (3,2,1,1,5)$.

For $p\geq 11$, if $x\geq 2$, then $-p^x+8p+4<0$. So, $x=1$ and $7p+4=z^2$. From $(z-2)(z+2)=7p$, we know $z-2=1$ is impossible. So, $z-2=7$ and $z+2=p$. Hence, we get a solution $(p,k,x,y,z)=(11,2,1,1,9)$.

Case 2: $k=1$ and $y=2$.

Equation \ref{eq3.4} can be written as $-p^x+(4p+2)^2=z^2$. We know $p=3,11,\dots$ since $p\equiv 3 \pmod 8$.

For $p=3$, we have  $196=3^x+z^2$. So, $x=3$ and $z=13$. We obtain a solution $(p,k,x,y,z)=(3,1,3,2,13)$.

For $p=11$, it is easy to check $46^2=11^x+z^2$ has no solution.

Since 19 is not a Sophie Germain prime, we can assume $p\geq 27$. 

If $x\geq 3$, then $p^x\geq p^3\geq 27p^2>(4p+2)^2$, so $-p^x+(4p+2)^2=z^2$ has no solutions. Since $x$ is odd, there is only one possibility, $x=1$. So, $-p+(4p+2)^2=z^2$, or $(4p+2-z)(4p+2+z)=p$. Hence, $4p+2+z=p$ which is impossible.
\end{proof}

\begin{remark}
From the proof of  Theorem \ref{th4}, It is clear that when $p\equiv 1, 5 \pmod 8$, for $k=0$, then equation in Theorem \ref{th4} has only solution $(x,y,z)=(0,0,0)$.
\end{remark}


\begin{thebibliography}{99}  

\bibitem {Acu}D. Acu,
\newblock\textquotedblleft On a Diophantine equation\textquotedblright, 
General Mathematics, Vol 15, NO.4 (2007), 145-148. 
 
\bibitem {Bor}P. B. Borah, M. Dutta,
\newblock\textquotedblleft On the Diophantine equation $7^x+32^y=z^2$ and its generalization\textquotedblright, 
Integer 22  (2022), A29. 

\bibitem {Bug0}Y. Bugeaud, M. Mignotte, Y. Roy,
\newblock\textquotedblleft On the Diophantine equation $\frac{x^n-1}{x-1}=y^q$\textquotedblright, 
Pacific J. Math, 193 (2000), 257-268. 


\bibitem {Bug}Y. Bugeaud, P. Mih$\breve{a}$ilescu,
\newblock\textquotedblleft On the Nagell-Ljunggren equation $\frac{x^n-1}{x-1}=y^q$\textquotedblright, 
Math Scand, 101 (2007), 177-183. 


\bibitem {Bur}N. Burshtein,
\newblock\textquotedblleft On the Diophantine equation $p^x+q^y=z^2$\textquotedblright, 
Annals of Pure and Applied Mathematics, Vol 13, NO. 2, (2017), 229-233. 

\bibitem {C}K.  Conrad,  Examples of Mordell's equation,
https://kconrad.math.uconn.edu/blurbs/

gradnumthy/mordelleqn1.pdf (accessed: 6 August 2022).


\bibitem {Gat} J. Gathen and I. E. Shparlinski,  \newblock\textquotedblleft Generating safe primes\textquotedblright,
J. Math. Cryptol. 7 (2013), 333 – 365.


\bibitem {Gay}W. S. Gayo, J. B. Bacani,
\newblock\textquotedblleft On the Diophantine equation $(M_p)^x+(M_q+1)^y=z^2$\textquotedblright, 
European Journal of Pure and Applied Mathematics, Vol. 14, No. 2, (2021), 396-403. 



\bibitem {Ko}C. Ko,
\newblock\textquotedblleft On the Diophantine equation $x^2=y^n+1$, $xy \neq 0$\textquotedblright, 
Sci. Sin., 14 (1965), 457-460. 

\bibitem {Li} Y. Li, M. H. Le,
\newblock\textquotedblleft On the Diophantine equation $\frac{x^m-1}{x-1}=y^n$\textquotedblright, 
Acta Arith. 73, No. 4, (1995),  363-366.



 \bibitem {Le}M. H. Le,
\newblock\textquotedblleft A note on the Diophantine equation $x^2+b^y=c^z$\textquotedblright, 
Acta Arith. 71, (1995), 253-257.

\bibitem {Lju}W. Ljunggren,
\newblock\textquotedblleft Noen Setninger om ubestemte likninger av formen $\frac{x^n-1}{x-1}=y^q$\textquotedblright, 
Norsk Mat. Tidsskr. 25 (1943), 17-20. 






\bibitem {LM}  The L-functions and modular forms database (LMFDB), 
https://www.lmfdb.org/EllipticCurve/Q/ (Accessed: 6 August  2022 ).


\bibitem {Mih}P. Mih$\breve{\mbox{a}}$ilescu,
\newblock\textquotedblleft Primary cycolotomic units and a proof of Catalan's conjecture\textquotedblright, 
J. Reine Angew Math. 572 (2004), 167-195.




\bibitem {Min2}R. J. S. Mina, J. B. Bacani,
\newblock\textquotedblleft On the solutions of the Diophantine equation $p^x+(p+4k)^y=z^2$ for prime pairs $p$ and $p+4k$\textquotedblright, 
European Journal of Pure and Applied Mathematics, Vol. 14, No. 2, (2021), 471-479. 


\bibitem {Mus}D. Musielak,
\newblock\textquotedblleft Germain and Her Fearless Attempt to Prove
Fermat’s Last Theorem\textquotedblright, 
	arXiv:1904.03553 [math.HO].

\bibitem {Sh} V. Shoup, \newblock\textquotedblleft A Computational Introduction to Number Theory and Algebra,\textquotedblright
 Cambridge University Press, 2009. 

\bibitem {Ter}N. Terai,
\newblock\textquotedblleft The Diophantine equation $x^2+q^m=p^n$ \textquotedblright, 
Acta Arith. 63 (1993), 351-358.

\bibitem {To} S. Tomita,    Diophantine Equation $5^x-y^2=4$,
 https://math.stackexchange.com/questions/

4184457/5x-y2-4-diophantine-equation.

\bibitem {Yua}P. Z. Yuan,
\newblock\textquotedblleft A note on the Diophantine equation $\frac{x^m-1}{x-1}=y^n$ (Chinese)\textquotedblright, 
Acta Math. Sinica (Chinese Ser.) 39, No. 2, (1996), 184-189.

 \bibitem {Yua2}P. Z. Yuan, J. B. Wang,
\newblock\textquotedblleft On the Diophantine equation $x^2+b^y=c^z$\textquotedblright, 
Acta Arith. 84, No. 2, (1998), 145-147.

\bibitem {Zhang}J. Zhang, Y. Li,
\newblock\textquotedblleft On equations $(-1)^{\alpha}p^x+(-1)^{\beta}(2^k(2p-1))^y=z^2$ for prime pair $p$ and $2p-1$\textquotedblright, 
Submitted.


\end{thebibliography}
\end{document}